\documentclass[12pt,fleqn]{article}

\usepackage{amsmath,amssymb,amsthm}
\usepackage{geometry}
\usepackage[pdfpagemode=UseNone,pdfstartview=FitH,hypertex]{hyperref}

\newcommand{\abs}[1]{\left|#1\right|}
\newcommand{\bdry}[1]{\partial #1}
\newcommand{\C}[1]{$(\text{C})_{#1}$}
\newcommand{\A}{{\cal A}}
\newcommand{\F}{{\cal F}}
\newcommand{\closure}[1]{\overline{#1}}
\newcommand{\dint}{\ds{\int}}
\newcommand{\dist}[2]{\text{dist}\, (#1,#2)}
\newcommand{\ds}[1]{\displaystyle #1}
\newcommand{\eps}{\varepsilon}
\newcommand{\id}[1][]{id_{\, #1}}
\newcommand{\incl}{\subset}
\newcommand{\loc}{\text{loc}}
\newcommand{\M}{{\cal M}}
\newcommand{\N}{\mathbb N}
\newcommand{\norm}[2][]{\left\|#2\right\|_{#1}}
\renewcommand{\O}{\text{O}}
\renewcommand{\o}{\text{o}}
\newcommand{\pnorm}[2][]{\if #1'' \left|#2\right|_p \else \left|#2\right|_{#1} \fi}
\newcommand{\R}{\mathbb R}
\newcommand{\RP}{\R \text{P}}
\newcommand{\restr}[2]{\left.#1\right|_{#2}}
\newcommand{\seq}[1]{\left(#1\right)}
\newcommand{\set}[1]{\left\{#1\right\}}
\newcommand{\wstar}{\xrightarrow{w^\ast}}
\newcommand{\Z}{\mathbb Z}
\DeclareMathOperator{\supp}{supp}

\newenvironment{enumroman}{\begin{enumerate}
\renewcommand{\theenumi}{$(\roman{enumi})$}
\renewcommand{\labelenumi}{$(\roman{enumi})$}}{\end{enumerate}}

\newenvironment{properties}[1]{\begin{enumerate}
\renewcommand{\theenumi}{$(#1_\arabic{enumi})$}
\renewcommand{\labelenumi}{$(#1_\arabic{enumi})$}}{\end{enumerate}}

\newtheorem{lemma}{Lemma}[section]
\newtheorem{proposition}[lemma]{Proposition}
\newtheorem{theorem}[lemma]{Theorem}

\theoremstyle{remark}
\newtheorem{example}[lemma]{Example}

\numberwithin{equation}{section}

\title{\bf Asymmetric critical $p$-Laplacian problems\thanks{Supported by NSFC (11325107, 11271353, 11501252, 11571176) and the Natural Science Foundation of Jiangsu Province of China for Young Scholars (No. BK2012109).}}
\author{\bf Kanishka Perera\\
{\small Department of Mathematical Sciences, Florida Institute}\\
{\small of Technology, Melbourne, FL 32901, USA. Email: kperera@fit.edu}\\
[\bigskipamount]
\bf Yang Yang\\
{\small School of Science, Jiangnan University, Wuxi,}\\
{\small Jiangsu, 214122, China. Email: yynjnu@126.com}\\
[\bigskipamount]
\bf Zhitao Zhang\footnote{Corresponding author.}\\
{\small Academy of Mathematics and Systems Science, Chinese Academy}\\
{\small of Sciences, Beijing, 100190, China. Email: zzt@math.ac.cn}}
\date{}

\begin{document}

\maketitle

\begin{abstract}
We obtain nontrivial solutions for two types of critical $p$-Laplacian problems with asymmetric nonlinearities in a smooth bounded domain in $\R^N,\, N \ge 2$. For $p < N$, we consider an asymmetric problem involving the critical Sobolev exponent $p^\ast = Np/(N - p)$. In the borderline case $p = N$, we consider an asymmetric critical exponential nonlinearity of the Trudinger-Moser type. In the absence of a suitable direct sum decomposition, we use a linking theorem based on the $\Z_2$-cohomological index to obtain our solutions.
\end{abstract}

\section{Introduction}

Beginning with the seminal paper of Ambrosetti and Prodi \cite{MR0320844}, elliptic boundary value problems with asymmetric nonlinearities have been extensively studied (see, e.g., Berger and Podolak \cite{MR0377274}, Kazdan and Warner \cite{MR0477445}, Dancer \cite{MR524624}, Amann and Hess \cite{MR549877}, and the references therein). More recently, Deng \cite{MR1137897}, de Figueiredo and Yang \cite{MR1758880}, Aubin and Wang \cite{MR1831984}, Calanchi and Ruf \cite{MR1938385}, and Zhang et al.\! \cite{MR2112476} have obtained interesting existence and multiplicity results for semilinear Ambrosetti-Prodi type problems with critical nonlinearities using variational methods.

In the present paper, first we consider the asymmetric critical $p$-Laplacian problem
\begin{equation} \label{1.1}
\left\{\begin{aligned}
- \Delta_p\, u & = \lambda\, |u|^{p-2}\, u + u_+^{p^\ast - 1} && \text{in } \Omega\\[10pt]
u & = 0 && \text{on } \bdry{\Omega},
\end{aligned}\right.
\end{equation}
where $\Omega$ is a smooth bounded domain in $\R^N,\, N \ge 2$, $1 < p < N$, $p^\ast = Np/(N - p)$ is the critical Sobolev exponent, $\lambda > 0$ is a constant, and $u_+(x) = \max \set{u(x),0}$. We recall that $\lambda \in \R$ is a Dirichlet eigenvalue of $- \Delta_p$ in $\Omega$ if the problem
\begin{equation} \label{1.2}
\left\{\begin{aligned}
- \Delta_p\, u & = \lambda\, |u|^{p-2}\, u && \text{in } \Omega\\[10pt]
u & = 0 && \text{on } \bdry{\Omega}
\end{aligned}\right.
\end{equation}
has a nontrivial solution. The first eigenvalue $\lambda_1(p)$ is positive, simple, and has an associated eigenfunction $\varphi_1$ that is positive in $\Omega$. Problem \eqref{1.1} has a positive solution when $N \ge p^2$ and $0 < \lambda < \lambda_1(p)$ (see Guedda and V{\'e}ron \cite{MR1009077}). When $\lambda = \lambda_1(p)$, $t \varphi_1$ is clearly a negative solution for any $t < 0$. Here we focus on the case $\lambda > \lambda_1(p)$. Our first result is the following.

\begin{theorem} \label{Theorem 1.1}
If $N \ge p^2$ and $\lambda > \lambda_1(p)$ is not an eigenvalue of $- \Delta_p$, then problem \eqref{1.1} has a nontrivial solution.
\end{theorem}

In the borderline case $p = N \ge 2$, critical growth is of exponential type and is governed by the Trudinger-Moser inequality
\begin{equation} \label{1.3}
\sup_{u \in W^{1,N}_0(\Omega),\; \norm{u} \le 1}\, \int_\Omega e^{\, \alpha_N\, |u|^{N'}} dx < \infty,
\end{equation}
where $\alpha_N = N \omega_{N-1}^{1/(N-1)}$, $\omega_{N-1}$ is the area of the unit sphere in $\R^N$, and $N' = N/(N - 1)$ (see Trudinger \cite{MR0216286} and Moser \cite{MR0301504}). A natural analog of problem \eqref{1.1} for this case is
\begin{equation} \label{1.4}
\left\{\begin{aligned}
- \Delta_N\, u & = \lambda\, |u|^{N-2}\, u\, e^{\, u_+^{N'}} && \text{in } \Omega\\[10pt]
u & = 0 && \text{on } \bdry{\Omega}.
\end{aligned}\right.
\end{equation}
A result of Adimurthi \cite{MR1079983} implies that this problem has a nonnegative and nontrivial solution when $0 < \lambda < \lambda_1(N)$ (see also do {\'O} \cite{MR1392090}). When $\lambda = \lambda_1(N)$, $t \varphi_1$ is again a negative solution for any $t < 0$. Our second result here is the following.

\begin{theorem} \label{Theorem 1.2}
If $N \ge 2$ and $\lambda > \lambda_1(N)$ is not an eigenvalue of $- \Delta_N$, then problem \eqref{1.4} has a nontrivial solution.
\end{theorem}

These results complement those in \cite{MR1831984,MR1938385,MR1758880,MR1137897,MR2112476} concerning the semilinear case $p = 2$. However, the linking arguments based on eigenspaces of $- \Delta$ used in those papers do not apply to the quasilinear case $p \ne 2$ since the nonlinear operator $- \Delta_p$ does not have linear eigenspaces. Therefore we will use more general constructions based on sublevel sets as in Perera and Szulkin \cite{MR2153141}. Moreover, the standard sequence of eigenvalues of $- \Delta_p$ based on the genus does not provide sufficient information about the structure of the sublevel sets to carry out these linking constructions, so we will use a different sequence of eigenvalues introduced in Perera \cite{MR1998432} that is based on a cohomological index.

The $\Z_2$-cohomological index of Fadell and Rabinowitz \cite{MR57:17677} is defined as follows. Let $W$ be a Banach space and let $\A$ denote the class of symmetric subsets of $W \setminus \set{0}$. For $A \in \A$, let $\overline{A} = A/\Z_2$ be the quotient space of $A$ with each $u$ and $-u$ identified, let $f : \overline{A} \to \RP^\infty$ be the classifying map of $\overline{A}$, and let $f^\ast : H^\ast(\RP^\infty) \to H^\ast(\overline{A})$ be the induced homomorphism of the Alexander-Spanier cohomology rings. The cohomological index of $A$ is defined by
\[
i(A) = \begin{cases}
0 & \text{if } A = \emptyset,\\[5pt]
\sup \set{m \ge 1 : f^\ast(\omega^{m-1}) \ne 0} & \text{if } A \ne \emptyset,
\end{cases}
\]
where $\omega \in H^1(\RP^\infty)$ is the generator of the polynomial ring $H^\ast(\RP^\infty) = \Z_2[\omega]$.

\begin{example}
The classifying map of the unit sphere $S^{m-1}$ in $\R^m,\, m \ge 1$ is the inclusion $\RP^{m-1} \incl \RP^\infty$, which induces isomorphisms on the cohomology groups $H^q$ for $q \le m - 1$, so $i(S^{m-1}) = m$.
\end{example}

The following proposition summarizes the basic properties of this index.

\begin{proposition}[Fadell-Rabinowitz \cite{MR57:17677}]
The index $i : \A \to \N \cup \set{0,\infty}$ has the following properties:
\begin{properties}{i}
\item Definiteness: $i(A) = 0$ if and only if $A = \emptyset$.
\item Monotonicity: If there is an odd continuous map from $A$ to $B$ (in particular, if $A \subset B$), then $i(A) \le i(B)$. Thus, equality holds when the map is an odd homeomorphism.
\item Dimension: $i(A) \le \dim W$.
\item Continuity: If $A$ is closed, then there is a closed neighborhood $N \in \A$ of $A$ such that $i(N) = i(A)$. When $A$ is compact, $N$ may be chosen to be a $\delta$-neighborhood $N_\delta(A) = \set{u \in W : \dist{u}{A} \le \delta}$.
\item Subadditivity: If $A$ and $B$ are closed, then $i(A \cup B) \le i(A) + i(B)$.
\item Stability: If $SA$ is the suspension of $A \ne \emptyset$, obtained as the quotient space of $A \times [-1,1]$ with $A \times \set{1}$ and $A \times \set{-1}$ collapsed to different points, then $i(SA) = i(A) + 1$.
\item Piercing property: If $A$, $A_0$ and $A_1$ are closed, and $\varphi : A \times [0,1] \to A_0 \cup A_1$ is a continuous map such that $\varphi(-u,t) = - \varphi(u,t)$ for all $(u,t) \in A \times [0,1]$, $\varphi(A \times [0,1])$ is closed, $\varphi(A \times \set{0}) \subset A_0$ and $\varphi(A \times \set{1}) \subset A_1$, then $i(\varphi(A \times [0,1]) \cap A_0 \cap A_1) \ge i(A)$.
\item Neighborhood of zero: If $U$ is a bounded closed symmetric neighborhood of $0$, then $i(\bdry{U}) = \dim W$.
\end{properties}
\end{proposition}

For $1 < p < \infty$, eigenvalues of problem \eqref{1.2} coincide with critical values of the functional
\[
\Psi(u) = \frac{1}{\dint_\Omega |u|^p\, dx}, \quad u \in \M = \set{u \in W^{1,p}_0(\Omega) : \int_\Omega |\nabla u|^p\, dx = 1}.
\]
Let $\F$ denote the class of symmetric subsets of $\M$ and set
\[
\lambda_k(p) := \inf_{M \in \F,\; i(M) \ge k}\, \sup_{u \in M}\, \Psi(u), \quad k \in \N.
\]
Then $0 < \lambda_1(p) < \lambda_2(p) \le \lambda_3(p) \le \cdots \to \infty$ is a sequence of eigenvalues of \eqref{1.2} and
\begin{equation} \label{1.5}
\lambda_k(p) < \lambda_{k+1}(p) \implies i(\Psi^{\lambda_k(p)}) = i(\M \setminus \Psi_{\lambda_{k+1}(p)}) = k,
\end{equation}
where $\Psi^a = \set{u \in \M : \Psi(u) \le a}$ and $\Psi_a = \set{u \in \M : \Psi(u) \ge a}$ for $a \in \R$ (see Perera et al.\! \cite[Propositions 3.52 and 3.53]{MR2640827}). As we will see, problems \eqref{1.1} and \eqref{1.4} have nontrivial solutions as long as $\lambda$ is not an eigenvalue from the sequence $\seq{\lambda_k(p)}$. This leaves an open question of existence of nontrivial solutions when $\lambda$ belongs to this sequence.

We will prove Theorems \ref{Theorem 1.1} and \ref{Theorem 1.2} using the following abstract critical point theorem proved in Yang and Perera \cite{YaPe2}, which generalizes the well-known linking theorem of Rabinowitz \cite{MR0488128}.

\begin{theorem} \label{Theorem 1.5}
Let $\Phi$ be a $C^1$-functional defined on a Banach space $W$ and let $A_0$ and $B_0$ be disjoint nonempty closed symmetric subsets of the unit sphere $S = \set{u \in W : \norm{u} = 1}$ such that
\[
i(A_0) = i(S \setminus B_0) < \infty.
\]
Assume that there exist $R > r > 0$ and $v \in S \setminus A_0$ such that
\[
\sup \Phi(A) \le \inf \Phi(B), \qquad \sup \Phi(X) < \infty,
\]
where
\begin{gather*}
A = \set{tu : u \in A_0,\, 0 \le t \le R} \cup \set{R\, \pi((1 - t)\, u + tv) : u \in A_0,\, 0 \le t \le 1},\\[10pt]
B = \set{ru : u \in B_0},\\[10pt]
X = \set{tu : u \in A,\, \norm{u} = R,\, 0 \le t \le 1},
\end{gather*}
and $\pi : W \setminus \set{0} \to S,\, u \mapsto u/\norm{u}$ is the radial projection onto $S$. Let $\Gamma = \{\gamma \in C(X,W) : \gamma(X) \text{ is closed and} \restr{\gamma}{A} = \id[\! A]\}$ and set
\[
c := \inf_{\gamma \in \Gamma}\, \sup_{u \in \gamma(X)}\, \Phi(u).
\]
Then
\begin{equation} \label{1.6}
\inf \Phi(B) \le c \le \sup \Phi(X),
\end{equation}
in particular, $c$ is finite. If, in addition, $\Phi$ satisfies the {\em \C{c}} condition, then $c$ is a critical value of $\Phi$.
\end{theorem}

This theorem was stated and proved under the Palais-Smale compactness condition in \cite{YaPe2}, but the proof goes through unchanged since the first deformation lemma also holds under the Cerami condition (see, e.g., Perera et al.\! \cite[Lemma 3.7]{MR2640827}). The linking construction used in the proof has also been used in Perera and Szulkin \cite{MR2153141} to obtain nontrivial solutions of $p$-Laplacian problems with nonlinearities that cross an eigenvalue. A similar construction based on the notion of cohomological linking was given in Degiovanni and Lancelotti \cite{MR2371112}. See also Perera et al.\! \cite[Proposition 3.23]{MR2640827}.

\section{Proof of Theorem \ref{Theorem 1.1}}

Weak solutions of problem \eqref{1.1} coincide with critical points of the $C^1$-functional
\[
\Phi(u) = \int_\Omega \left[\frac{1}{p}\, \big(|\nabla u|^p - \lambda\, |u|^p\big) - \frac{1}{p^\ast}\, u_+^{p^\ast}\right] dx, \quad u \in W^{1,p}_0(\Omega).
\]
We recall that $\Phi$ satisfies the Cerami compactness condition at the level $c \in \R$, or the \C{c} condition for short, if every sequence $\seq{u_j} \subset W^{1,p}_0(\Omega)$ such that $\Phi(u_j) \to c$ and $\left(1 + \norm{u_j}\right) \Phi'(u_j) \to 0$, called a \C{c} sequence, has a convergent subsequence. Let
\begin{equation} \label{2.1}
S = \inf_{u \in W^{1,p}_0(\Omega) \setminus \set{0}}\, \frac{\dint_\Omega |\nabla u|^p\, dx}{\left(\dint_\Omega |u|^{p^\ast}\, dx\right)^{p/p^\ast}}
\end{equation}
be the best constant in the Sobolev inequality.

\begin{lemma} \label{Lemma 2.1}
If $\lambda \ne \lambda_1(p)$, then $\Phi$ satisfies the {\em \C{c}} condition for all $c < \dfrac{1}{N}\, S^{N/p}$.
\end{lemma}

\begin{proof}
Let $c < \dfrac{1}{N}\, S^{N/p}$ and let $\seq{u_j}$ be a \C{c} sequence. First we show that $\seq{u_j}$ is bounded. We have
\begin{equation} \label{2.2}
\int_\Omega \left[\frac{1}{p}\, \big(|\nabla u_j|^p - \lambda\, |u_j|^p\big) - \frac{1}{p^\ast}\, u_{j+}^{p^\ast}\right] dx = c + \o(1)
\end{equation}
and
\begin{equation} \label{2.3}
\int_\Omega \big(|\nabla u_j|^{p-2}\, \nabla u_j \cdot \nabla v - \lambda\, |u_j|^{p-2}\, u_j\, v - u_{j+}^{p^\ast - 1}\, v\big)\, dx = \frac{\o(1) \norm{v}}{1 + \norm{u_j}}, \quad \forall v \in W^{1,p}_0(\Omega).
\end{equation}
Taking $v = u_j$ in \eqref{2.3} and combining with \eqref{2.2} gives
\begin{equation} \label{2.4}
\int_\Omega u_{j+}^{p^\ast}\, dx = Nc + \o(1),
\end{equation}
and taking $v = u_{j+}$ in \eqref{2.3} gives
\[
\int_\Omega |\nabla u_{j+}|^p\, dx = \int_\Omega \big(\lambda\, u_{j+}^p + u_{j+}^{p^\ast}\big)\, dx + \o(1),
\]
so $\seq{u_{j+}}$ is bounded in $W^{1,p}_0(\Omega)$. Suppose $\rho_j := \norm{u_j} \to \infty$ for a renamed subsequence. Then $\tilde{u}_j := u_j/\rho_j$ converges to some $\tilde{u}$ weakly in $W^{1,p}_0(\Omega)$, strongly in $L^q(\Omega)$ for $1 \le q < p^\ast$, and a.e.\! in $\Omega$ for a further subsequence. Since the sequence $\seq{u_{j+}}$ is bounded, dividing \eqref{2.2} by $\rho_j^p$ and \eqref{2.3} by $\rho_j^{p-1}$, and passing to the limit then gives
\[
1 = \lambda \int_\Omega |\tilde{u}|^p\, dx, \qquad \int_\Omega |\nabla \tilde{u}|^{p-2}\, \nabla \tilde{u} \cdot \nabla v\, dx = \lambda \int_\Omega |\tilde{u}|^{p-1}\, \tilde{u}\, v\, dx \quad \forall v \in W^{1,p}_0(\Omega),
\]
respectively. Moreover, since $\tilde{u}_{j+} = u_{j+}/\rho_j \to 0$, $\tilde{u} \le 0$ a.e. Hence $\tilde{u} = t \varphi_1$ for some $t < 0$ and $\lambda = \lambda_1(p)$, contrary to assumption.

Since $\seq{u_j}$ is bounded, so is $\seq{u_{j+}}$, a renamed subsequence of which then converges to some $v \ge 0$ weakly in $W^{1,p}_0(\Omega)$, strongly in $L^q(\Omega)$ for $1 \le q < p^\ast$ and a.e.\! in $\Omega$, and
\begin{equation} \label{2.5}
|\nabla u_{j+}|^p\, dx \wstar \mu, \qquad u_{j+}^{p^\ast}\, dx \wstar \nu
\end{equation}
in the sense of measures, where $\mu$ and $\nu$ are bounded nonnegative measures on $\closure{\Omega}$ (see, e.g., Folland \cite{MR1681462}). By the concentration compactness principle of Lions \cite{MR834360,MR850686}, then there exist an at most countable index set $I$ and points $x_i \in \closure{\Omega},\, i \in I$ such that
\begin{equation} \label{2.6}
\mu \ge |\nabla v|^p\, dx + \sum_{i \in I} \mu_i\, \delta_{x_i}, \qquad \nu = v^{p^\ast}\, dx + \sum_{i \in I} \nu_i\, \delta_{x_i},
\end{equation}
where $\mu_i, \nu_i > 0$ and $\nu_i^{p/p^\ast} \le \mu_i/S$. Let $\varphi : \R^N \to [0,1]$ be a smooth function such that $\varphi(x) = 1$ for $|x| \le 1$ and $\varphi(x) = 0$ for $|x| \ge 2$. Then set
\[
\varphi_{i,\rho}(x) = \varphi\left(\frac{x - x_i}{\rho}\right), \quad x \in \R^N
\]
for $i \in I$ and $\rho > 0$, and note that $\varphi_{i,\rho} : \R^N \to [0,1]$ is a smooth function such that $\varphi_{i,\rho}(x) = 1$ for $|x - x_i| \le \rho$ and $\varphi_{i,\rho}(x) = 0$ for $|x - x_i| \ge 2 \rho$. The sequence $\seq{\varphi_{i,\rho}\, u_{j+}}$ is bounded in $W^{1,p}_0(\Omega)$ and hence taking $v = \varphi_{i,\rho}\, u_{j+}$ in \eqref{2.3} gives
\begin{equation} \label{2.7}
\int_\Omega \big(\varphi_{i,\rho}\, |\nabla u_{j+}|^p + u_{j+}\, |\nabla u_{j+}|^{p-2}\, \nabla u_{j+} \cdot \nabla \varphi_{i,\rho} - \lambda\, \varphi_{i,\rho}\, u_{j+}^p - \varphi_{i,\rho}\, u_{j+}^{p^\ast}\big)\, dx = \o(1).
\end{equation}
By \eqref{2.5},
\[
\int_\Omega \varphi_{i,\rho}\, |\nabla u_{j+}|^p\, dx \to \int_\Omega \varphi_{i,\rho}\, d\mu, \qquad \int_\Omega \varphi_{i,\rho}\, u_{j+}^{p^\ast}\, dx \to \int_\Omega \varphi_{i,\rho}\, d\nu.
\]
Denoting by $C$ a generic positive constant independent of $j$ and $\rho$,
\[
\abs{\int_\Omega \big(u_{j+}\, |\nabla u_{j+}|^{p-2}\, \nabla u_{j+} \cdot \nabla \varphi_{i,\rho} - \lambda\, \varphi_{i,\rho}\, u_{j+}^p\big)\, dx} \le C \left(\frac{I_j^{1/p}}{\rho} + I_j\right),
\]
where
\[
I_j := \int_{\Omega \cap B_{2 \rho}(x_i)} u_{j+}^p\, dx \to \int_{\Omega \cap B_{2 \rho}(x_i)} v^p\, dx \le C \rho^p \left(\int_{\Omega \cap B_{2 \rho}(x_i)} v^{p^\ast}\, dx\right)^{p/p^\ast}.
\]
So passing to the limit in \eqref{2.7} gives
\[
\int_\Omega \varphi_{i,\rho}\, d\mu - \int_\Omega \varphi_{i,\rho}\, d\nu \le C \left[\left(\int_{\Omega \cap B_{2 \rho}(x_i)} v^{p^\ast}\, dx\right)^{1/p^\ast} + \int_{\Omega \cap B_{2 \rho}(x_i)} v^p\, dx\right].
\]
Letting $\rho \searrow 0$ and using \eqref{2.6} now gives $\mu_i \le \nu_i$, which together with $\nu_i^{p/p^\ast} \le \mu_i/S$ then gives $\nu_i = 0$ or $\nu_i \ge S^{N/p}$. Passing to the limit in \eqref{2.4} and using \eqref{2.5} and \eqref{2.6} gives $\nu_i \le Nc < S^{N/p}$, so $\nu_i = 0$. Hence $I = \emptyset$ and
\begin{equation} \label{2.8}
\int_\Omega u_{j+}^{p^\ast}\, dx \to \int_\Omega v^{p^\ast}\, dx.
\end{equation}

Passing to a further subsequence, $u_j$ converges to some $u$ weakly in $W^{1,p}_0(\Omega)$, strongly in $L^q(\Omega)$ for $1 \le q < p^\ast$, and a.e.\! in $\Omega$. Since
\[
|u_{j+}^{p^\ast - 1}\, (u_j - u)| \le u_{j+}^{p^\ast} + u_{j+}^{p^\ast - 1}\, |u| \le \left(2 - \frac{1}{p^\ast}\right) u_{j+}^{p^\ast} + \frac{1}{p^\ast}\, |u|^{p^\ast}
\]
by Young's inequality,
\[
\int_\Omega u_{j+}^{p^\ast - 1}\, (u_j - u)\, dx \to 0
\]
by \eqref{2.8} and the dominated convergence theorem. Then $u_j \to u$ in $W^{1,p}_0(\Omega)$ by a standard argument.
\end{proof}

We recall that the infimum in \eqref{2.1} is attained by the family of functions
\[
u_\eps(x) = \frac{C_{N,p}\, \eps^{-(N-p)/p}}{\left[1 + \left(\dfrac{|x|}{\eps}\right)^{p/(p-1)}\right]^{(N-p)/p}}, \quad \eps > 0
\]
when $\Omega = \R^N$, where $C_{N,p} > 0$ is chosen so that
\[
\int_{\R^N} |\nabla u_\eps|^p\, dx = \int_{\R^N} u_\eps^{p^\ast}\, dx = S^{N/p}.
\]
Take a smooth function $\eta : [0,\infty) \to [0,1]$ such that $\eta(s) = 1$ for $s \le 1/4$ and $\eta(s) = 0$ for $s \ge 1/2$, and set
\[
u_{\eps,\delta}(x) = \eta\!\left(\frac{|x|}{\delta}\right) u_\eps(x), \quad \eps, \delta > 0.
\]
We have the well-known estimates
\begin{gather}
\label{2.9} \int_{\R^N} |\nabla u_{\eps,\delta}|^p\, dx \le S^{N/p} + C \left(\frac{\eps}{\delta}\right)^{(N-p)/(p-1)},\\[10pt]
\label{2.10} \int_{\R^N} u_{\eps,\delta}^{p^\ast}\, dx \ge S^{N/p} - C \left(\frac{\eps}{\delta}\right)^{N/(p-1)},\\[10pt]
\label{2.11} \int_{\R^N} u_{\eps,\delta}^p\, dx \ge \begin{cases}
\dfrac{\eps^p}{C} - C \delta^p \left(\dfrac{\eps}{\delta}\right)^{(N-p)/(p-1)} & \text{if } N > p^2\\[10pt]
\dfrac{\eps^p}{C}\, \log\!\left(\dfrac{\delta}{\eps}\right) - C \eps^p & \text{if } N = p^2,
\end{cases}
\end{gather}
where $C = C(N,p) > 0$ is a constant (see, e.g., Degiovanni and Lancelotti \cite{MR2514055}).

Let $i$, $\M$, $\Psi$, and $\lambda_k(p)$ be as in the introduction, and suppose that $\lambda_k(p) < \lambda_{k+1}(p)$. Then the sublevel set $\Psi^{\lambda_k(p)}$ has a compact symmetric subset $E$ of index $k$ that is bounded in $L^\infty(\Omega) \cap C^{1,\alpha}_\loc(\Omega)$ (see \cite[Theorem 2.3]{MR2514055}). We may assume without loss of generality that $0 \in \Omega$. Let $\delta_0 = \dist{0}{\bdry{\Omega}}$, take a smooth function $\theta : [0,\infty) \to [0,1]$ such that $\theta(s) = 0$ for $s \le 3/4$ and $\theta(s) = 1$ for $s \ge 1$, set
\[
v_\delta(x) = \theta\!\left(\frac{|x|}{\delta}\right) v(x), \quad v \in E,\, 0 < \delta \le \frac{\delta_0}{2},
\]
and let $E_\delta = \set{\pi(v_\delta) : v \in E}$, where $\pi : W^{1,p}_0(\Omega) \setminus \set{0} \to \M,\, u \mapsto u/\norm{u}$ is the radial projection onto $\M$.

\begin{lemma} \label{Lemma 2.2}
There exists a constant $C = C(N,p,\Omega,k) > 0$ such that for all sufficiently small $\delta > 0$,
\begin{enumroman}
\item $\Psi(w) \le \lambda_k(p) + C \delta^{N-p} \quad \forall w \in E_\delta$,
\item \label{Lemma 2.2.ii} $E_\delta \cap \Psi_{\lambda_{k+1}(p)} = \emptyset$,
\item \label{Lemma 2.2.iii} $i(E_\delta) = k$,
\item \label{Lemma 2.2.iv} $\supp w \cap \supp \pi(u_{\eps,\delta}) = \emptyset \quad \forall w \in E_\delta$,
\item \label{Lemma 2.2.v} $\pi(u_{\eps,\delta}) \notin E_\delta$.
\end{enumroman}
\end{lemma}

\begin{proof}
Let $v \in E$ and let $w = \pi(v_\delta)$. We have
\[
\int_\Omega |\nabla v_\delta|^p\, dx \le \int_{\Omega \setminus B_\delta(0)} |\nabla v|^p\, dx + C \int_{B_\delta(0)} \left(|\nabla v|^p + \frac{|v|^p}{\delta^p}\right) dx \le 1 + C \delta^{N-p}
\]
since $E \subset \M$ is bounded in $C^1(B_{\delta_0/2}(0))$, and
\[
\int_\Omega |v_\delta|^p\, dx \ge \int_{\Omega \setminus B_\delta(0)} |v|^p\, dx = \int_\Omega |v|^p\, dx - \int_{B_\delta(0)} |v|^p\, dx \ge \frac{1}{\lambda_k(p)} - C \delta^N
\]
since $E \subset \Psi^{\lambda_k(p)}$, so
\[
\Psi(w) = \frac{\dint_\Omega |\nabla v_\delta|^p\, dx}{\dint_\Omega |v_\delta|^p\, dx} \le \lambda_k(p) + C \delta^{N-p}
\]
if $\delta > 0$ is sufficiently small. Taking $\delta$ so small that $\lambda_k(p) + C \delta^{N-p} < \lambda_{k+1}(p)$ then gives \ref{Lemma 2.2.ii}. Since $E_\delta \subset \M \setminus \Psi_{\lambda_{k+1}(p)}$ by \ref{Lemma 2.2.ii},
\[
i(E_\delta) \le i(\M \setminus \Psi_{\lambda_{k+1}(p)}) = k
\]
by the monotonicity of the index and \eqref{1.5}. On the other hand, since $E \to E_\delta,\, v \mapsto \pi(v_\delta)$ is an odd continuous map,
\[
i(E_\delta) \ge i(E) = k.
\]
So $i(E_\delta) = k$.

Since $\supp w = \supp v_\delta \subset \Omega \setminus B_{3 \delta/4}(0)$ and $\supp \pi(u_{\eps,\delta}) = \supp u_{\eps,\delta} \subset \closure{B_{\delta/2}(0)}$, \ref{Lemma 2.2.iv} is clear, and \ref{Lemma 2.2.v} is immediate from \ref{Lemma 2.2.iv}.
\end{proof}

We are now ready to prove Theorem \ref{Theorem 1.1}.

\begin{proof}[\bf Proof of Theorem \ref{Theorem 1.1}]
We have $\lambda_k(p) < \lambda < \lambda_{k+1}(p)$ for some $k \in \N$. Fix $\lambda_k(p) < \lambda' < \lambda$ and $\delta > 0$ so small that the conclusions of Lemma \ref{Lemma 2.2} hold with $\lambda_k(p) + C \delta^{N-p} \le \lambda'$, in particular,
\begin{equation} \label{2.12}
\Psi(w) \le \lambda' \quad \forall w \in E_\delta.
\end{equation}
Then take $A_0 = E_\delta$ and $B_0 = \Psi_{\lambda_{k+1}(p)}$, and note that $A_0$ and $B_0$ are disjoint nonempty closed symmetric subsets of $\M$ such that
\[
i(A_0) = i(\M \setminus B_0) = k
\]
by Lemma \ref{Lemma 2.2} \ref{Lemma 2.2.iii} and \eqref{1.5}. Now let $R > r > 0$, let $v_0 = \pi(u_{\eps,\delta})$, which is in $\M \setminus E_\delta$ by Lemma \ref{Lemma 2.2} \ref{Lemma 2.2.v}, and let $A$, $B$, and $X$ be as in Theorem \ref{Theorem 1.5}.

For $u \in \Psi_{\lambda_{k+1}(p)}$,
\[
\Phi(ru) \ge \frac{1}{p} \left(1 - \frac{\lambda}{\lambda_{k+1}(p)}\right) r^p - \frac{1}{p^\ast\, S^{p^\ast/p}}\, r^{p^\ast}
\]
by \eqref{2.1}. Since $\lambda < \lambda_{k+1}(p)$, it follows that $\inf \Phi(B) > 0$ if $r$ is sufficiently small. Next we show that $\Phi \le 0$ on $A$ if $R$ is sufficiently large. For $w \in E_\delta$ and $t \ge 0$,
\[
\Phi(tw) \le \frac{t^p}{p} \left(1 - \frac{\lambda}{\Psi(w)}\right) \le - \frac{t^p}{p} \left(\frac{\lambda}{\lambda'} - 1\right) \le 0
\]
by \eqref{2.12}. Now let $0 \le t \le 1$ and set $u = \pi((1 - t)\, w + tv_0)$. Since
\[
\norm{(1 - t)\, w + tv_0} \le (1 - t) \norm{w} + t \norm{v_0} = 1
\]
and since the supports of $w$ and $v_0 \ge 0$ are disjoint by Lemma \ref{Lemma 2.2} \ref{Lemma 2.2.iv},
\[
\pnorm{u}^p = \frac{\pnorm{(1 - t)\, w + tv_0}^p}{\norm{(1 - t)\, w + tv_0}^p} \ge (1 - t)^p \pnorm{w}^p + t^p \pnorm{v_0}^p \ge \frac{(1 - t)^p}{\Psi(w)} \ge \frac{(1 - t)^p}{\lambda'}
\]
by \eqref{2.12}, and
\[
\pnorm[p^\ast]{u_+}^{p^\ast} = \frac{\pnorm[p^\ast]{[(1 - t)\, w + tv_0]_+}^{p^\ast}}{\norm{(1 - t)\, w + tv_0}^{p^\ast}} \ge (1 - t)^{p^\ast} \pnorm[p^\ast]{w_+}^{p^\ast} + t^{p^\ast} \pnorm[p^\ast]{v_0}^{p^\ast} \ge t^{p^\ast}\, \frac{\pnorm[p^\ast]{u_{\eps,\delta}}^{p^\ast}}{\norm{u_{\eps,\delta}}^{p^\ast}} \ge \frac{t^{p^\ast}}{C}
\]
by \eqref{2.9} and \eqref{2.10} if $\eps$ is sufficiently small, where $C = C(N,p,\Omega,k) > 0$. Then
\[
\Phi(Ru) = \frac{R^p}{p} \norm{u}^p - \frac{\lambda R^p}{p} \pnorm{u}^p - \frac{R^{p^\ast}}{p^\ast} \pnorm[p^\ast]{u_+}^{p^\ast} \le - \frac{1}{p} \left[\frac{\lambda}{\lambda'}\, (1 - t)^p - 1\right] R^p - \frac{t^{p^\ast}}{C}\, R^{p^\ast}.
\]
The last expression is clearly nonpositive if $t \le 1 - (\lambda'/\lambda)^{1/p} =: t_0$. For $t > t_0$, it is nonpositive if $R$ is sufficiently large.

Now we show that $\sup \Phi(X) < \dfrac{1}{N}\, S^{N/p}$ if $\eps$ is sufficiently small. Noting that
\[
X = \set{\rho\, \pi((1 - t)\, w + tv_0) : w \in E_\delta,\, 0 \le t \le 1,\, 0 \le \rho \le R},
\]
let $w \in E_\delta$, let $0 \le t \le 1$, and set $u = \pi((1 - t)\, w + tv_0)$. Then
\[
\sup_{0 \le \rho \le R}\, \Phi(\rho u) \le \sup_{\rho \ge 0}\, \left[\frac{\rho^p}{p}\, \big(1 - \lambda \pnorm{u}^p\big) - \frac{\rho^{p^\ast}}{p^\ast} \pnorm[p^\ast]{u_+}^{p^\ast}\right] = \frac{1}{N}\, S_u(\lambda)^{N/p}
\]
when $1 - \lambda \pnorm{u}^p > 0$, where
\begin{multline*}
S_u(\lambda) = \frac{1 - \lambda \pnorm{u}^p}{\pnorm[p^\ast]{u_+}^p} = \frac{\norm{(1 - t)\, w + tv_0}^p - \lambda \pnorm{(1 - t)\, w + tv_0}^p}{\pnorm[p^\ast]{[(1 - t)\, w + tv_0]_+}^p}\\[10pt]
= \frac{(1 - t)^p\, \big(\norm{w}^p - \lambda \pnorm{w}^p\big) + t^p\, \big(\norm{v_0}^p - \lambda \pnorm{v_0}^p\big)}{\left[(1 - t)^{p^\ast} \pnorm[p^\ast]{w_+}^{p^\ast} + t^{p^\ast} \pnorm[p^\ast]{v_0}^{p^\ast}\right]^{p/p^\ast}}.
\end{multline*}
Since $\norm{w}^p - \lambda \pnorm{w}^p = 1 - \lambda/\Psi(w) \le 0$ by \eqref{2.12},
\[
S_u(\lambda) \le \frac{1 - \lambda \pnorm{v_0}^p}{\pnorm[p^\ast]{v_0}^p} = \frac{\norm{u_{\eps,\delta}}^p - \lambda \pnorm{u_{\eps,\delta}}^p}{\pnorm[p^\ast]{u_{\eps,\delta}}^p} \le \begin{cases}
S - \dfrac{\eps^p}{C} + C \eps^{(N-p)/(p-1)} & \text{if } N > p^2\\[10pt]
S - \dfrac{\eps^p}{C}\, |\!\log \eps| + C \eps^p & \text{if } N = p^2
\end{cases}
\]
by \eqref{2.9}--\eqref{2.11}. In both cases the last expression is strictly less than $S$ if $\eps$ is sufficiently small.

The inequalities \eqref{1.6} now imply that $0 < c < \dfrac{1}{N}\, S^{N/p}$. Then $\Phi$ satisfies the \C{c} condition by Lemma \ref{Lemma 2.1} and hence $c$ is a critical value of $\Phi$ by Theorem \ref{Theorem 1.5}.
\end{proof}

\section{Proof of Theorem \ref{Theorem 1.2}}

Weak solutions of problem \eqref{1.4} coincide with critical points of the $C^1$-functional
\[
\Phi(u) = \int_\Omega \left[\frac{1}{N}\, |\nabla u|^N - \lambda\, F(u)\right] dx, \quad u \in W^{1,N}_0(\Omega),
\]
where
\[
F(t) = \int_0^t |s|^{N-2}\, s\, e^{\, s_+^{N'}} ds.
\]
First we obtain some estimates for the primitive $F$.

\begin{lemma}
For all $t \in \R$,
\begin{gather}
\label{3.1} F(t) \le \frac{t_+^N}{2N}\, e^{\, t_+^{N'}} + \frac{t_-^N}{N} + C,\\[10pt]
\label{3.2} F(t) \le |t|^{N-1}\, e^{\, t_+^{N'}} + \frac{t_-^N}{N} + C,
\end{gather}
where $C$ denotes a generic positive constant and $t_- = \max \set{-t,0}$.
\end{lemma}

\begin{proof}
For $t \le 0$, $F(t) = |t|^N/N$. For $t > 0$, integrating by parts gives
\[
F(t) = \int_0^t s^{N-1}\, e^{\, s^{N'}} ds = \frac{t^N}{N}\, e^{\, t^{N'}} - \frac{N'}{N} \int_0^t s^{N+N'-1}\, e^{\, s^{N'}} ds.
\]
For $t \ge (N/N')^{1/N'}$, the last term is greater than or equal to
\[
\frac{N'}{N} \int_{(N/N')^{1/N'}}^t s^{N+N'-1}\, e^{\, s^{N'}} ds \ge \int_{(N/N')^{1/N'}}^t s^{N-1}\, e^{\, s^{N'}} ds = F(t) - F((N/N')^{1/N'})
\]
and hence
\[
2F(t) \le \frac{t^N}{N}\, e^{\, t^{N'}} + F((N/N')^{1/N'}).
\]
Since $F$ is bounded on bounded sets, \eqref{3.1} follows. As for \eqref{3.2}, $F(t) = (e^{t^2} - 1)/2$ for $t > 0$ if $N = 2$, and
\[
F(t) = \frac{t^{N-N'}}{N'}\, e^{\, t^{N'}} - \frac{N - N'}{N'} \int_0^t s^{N-N'-1}\, e^{\, s^{N'}} ds \le t^{N-1}\, e^{\, t^{N'}}
\]
for $t \ge 1/(N')^{1/(N'-1)}$ if $N \ge 3$.
\end{proof}

Proof of Theorem \ref{Theorem 1.2} will be based on the following lemma.

\begin{lemma} \label{Lemma 3.2}
If $\lambda \ne \lambda_1(N)$ and $0 \ne c < \alpha_N^{N-1}/N$, then every {\em \C{c}} sequence has a subsequence that converges weakly to a nontrivial critical point of $\Phi$.
\end{lemma}

\begin{proof}
Let $\lambda \ne \lambda_1(N)$, let $0 \ne c < \alpha_N^{N-1}/N$, and let $\seq{u_j}$ be a \C{c} sequence. First we show that $\seq{u_j}$ is bounded. We have
\begin{equation} \label{3.3}
\int_\Omega \left[\frac{1}{N}\, |\nabla u_j|^N - \lambda\, F(u_j)\right] dx = c + \o(1)
\end{equation}
and
\begin{equation} \label{3.4}
\int_\Omega \left(|\nabla u_j|^{N-2}\, \nabla u_j \cdot \nabla v - \lambda\, |u_j|^{N-2}\, u_j\, e^{\, u_{j+}^{N'}}\, v\right) dx = \frac{\o(1) \norm{v}}{1 + \norm{u_j}} \quad \forall v \in W^{1,N}_0(\Omega),
\end{equation}
in particular,
\begin{equation} \label{3.5}
\int_\Omega \left(|\nabla u_j|^N - \lambda\, |u_j|^N\, e^{\, u_{j+}^{N'}}\right) dx = \o(1).
\end{equation}
Combining \eqref{3.5} with \eqref{3.3} and \eqref{3.1} gives
\begin{equation} \label{3.6}
\int_\Omega u_{j+}^N\, e^{\, u_{j+}^{N'}}\, dx \le C,
\end{equation}
and taking $v = u_{j+}$ in \eqref{3.4} gives
\[
\int_\Omega |\nabla u_{j+}|^N\, dx = \lambda \int_\Omega u_{j+}^N\, e^{\, u_{j+}^{N'}}\, dx + \o(1),
\]
so the sequence $\seq{u_{j+}}$ is bounded in $W^{1,N}_0(\Omega)$. Passing to a subsequence, $u_{j+}$ then converges to some $\hat{u} \ge 0$ weakly in $W^{1,N}_0(\Omega)$, strongly in $L^q(\Omega)$ for $1 \le q < \infty$, and a.e.\! in $\Omega$. Then for any $v \in C^\infty_0(\Omega)$,
\begin{equation} \label{3.7}
\int_\Omega u_{j+}^{N-1}\, e^{\, u_{j+}^{N'}}\, v\, dx \to \int_\Omega \hat{u}^{N-1}\, e^{\, \hat{u}^{N'}} v\, dx
\end{equation}
by de Figueiredo et al.\! \cite[Lemma 2.1]{MR1386960} and \eqref{3.6}. Now suppose $\rho_j := \norm{u_{j-}} \to \infty$. Then $\tilde{u}_j := u_{j-}/\rho_j$ converges to some $\tilde{u} \ge 0$ weakly in $W^{1,N}_0(\Omega)$, strongly in $L^q(\Omega)$ for $1 \le q < \infty$, and a.e.\! in $\Omega$ for a further subsequence. Taking $v = u_{j-}$ in \eqref{3.4}, dividing by $\rho_j^N$, and passing to the limit then gives
\[
1 = \lambda \int_\Omega \tilde{u}^N\, dx,
\]
so $\tilde{u} \ne 0$. Since the sequence $\seq{u_{j+}}$ is bounded, dividing \eqref{3.4} by $\rho_j^{N-1}$ gives
\[
\int_\Omega |\nabla \tilde{u}_j|^{N-2}\, \nabla \tilde{u}_j \cdot \nabla v\, dx = \lambda \int_\Omega \tilde{u}_j^{N-1}\, v\, dx - \frac{\lambda}{\rho_j^{N-1}} \int_\Omega u_{j+}^{N-1}\, e^{\, u_{j+}^{N'}}\, v\, dx + \o(1),
\]
and passing to the limit using \eqref{3.7} gives
\[
\int_\Omega |\nabla \tilde{u}|^{N-2}\, \nabla \tilde{u} \cdot \nabla v\, dx = \lambda \int_\Omega \tilde{u}^{N-1}\, v\, dx \quad \forall v \in C^\infty_0(\Omega).
\]
This then holds for all $v \in W^{1,N}_0(\Omega)$ by density, so $\tilde{u} = t \varphi_1$ for some $t > 0$ and $\lambda = \lambda_1(N)$, contrary to assumption.

Since the sequence $\seq{u_j}$ is bounded, a renamed subsequence converges to some $u$ weakly in $W^{1,N}_0(\Omega)$, strongly in $L^q(\Omega)$ for $1 \le q < \infty$, and a.e.\! in $\Omega$. Since $\int_\Omega |u_j|^N\, e^{\, u_{j+}^{N'}}\, dx$ is bounded by \eqref{3.5}, then for any $v \in C^\infty_0(\Omega)$,
\[
\int_\Omega |u_j|^{N-2}\, u_j\, e^{\, u_{j+}^{N'}}\, v\, dx \to \int_\Omega |u|^{N-2}\, u\, e^{\, u_+^{N'}} v\, dx
\]
by de Figueiredo et al.\! \cite[Lemma 2.1]{MR1386960}. So passing to the limit in \eqref{3.4} gives
\[
\int_\Omega \left(|\nabla u|^{N-2}\, \nabla u \cdot \nabla v - \lambda\, |u|^{N-2}\, u\, e^{\, u_+^{N'}} v\right) dx = 0.
\]
This then holds for all $v \in W^{1,N}_0(\Omega)$ by density, so $u$ is a critical point of $\Phi$.

Suppose $u = 0$. Then
\[
\int_\Omega |u_j|^{N-1}\, e^{\, u_{j+}^{N'}}\, dx \to 0
\]
by de Figueiredo et al.\! \cite[Lemma 2.1]{MR1386960} as above, and hence
\[
\int_\Omega F(u_j)\, dx \to 0
\]
by \eqref{3.2} and the dominated convergence theorem, so
\[
\int_\Omega |\nabla u_j|^N\, dx \to Nc
\]
by \eqref{3.3}. Since $c < \alpha_N^{N-1}/N$, then $\limsup\, \norm{u_j} < \alpha_N^{1/N'}$, so there exists $\beta > 1/\alpha_N^{1/N'}$ such that $\beta \norm{u_j} \le 1$ for all sufficiently large $j$. For $1 < \gamma < \infty$ given by $1/\alpha_N \beta^{N'} + 1/\gamma = 1$, then
\[
\int_\Omega |u_j|^N\, e^{\, u_{j+}^{N'}}\, dx \le \left(\int_\Omega |u_j|^{\gamma N}\, dx\right)^{1/\gamma} \left(\int_\Omega e^{\, \alpha_N\, (\beta u_{j+})^{N'}} dx\right)^{1/\alpha_N \beta^{N'}} \to 0
\]
since $u_j \to 0$ in $L^{\gamma N}(\Omega)$ and the last integral is bounded by \eqref{1.3}. Then $u_j \to 0$ in $W^{1,N}_0(\Omega)$ by \eqref{3.5}, so $\Phi(u_j) \to 0$, contradicting $c \ne 0$.
\end{proof}

Let $i$, $\M$, $\Psi$, and $\lambda_k(N)$ be as in the introduction, and suppose that $\lambda_k(N) < \lambda_{k+1}(N)$. Then the sublevel set $\Psi^{\lambda_k(N)}$ has a compact symmetric subset $E$ of index $k$ that is bounded in $L^\infty(\Omega) \cap C^{1,\alpha}_\loc(\Omega)$ (see Degiovanni and Lancelotti \cite[Theorem 2.3]{MR2514055}). We may assume without loss of generality that $0 \in \Omega$. For all $m \in \N$ so large that $B_{2/m}(0) \subset \Omega$, let
\[
\eta_m(x) = \begin{cases}
0 & \text{if } |x| \le 1/2\, m^{m+1}\\[5pt]
2\, m^m \left(|x| - \dfrac{1}{2\, m^{m+1}}\right) & \text{if } 1/2\, m^{m+1} < |x| \le 1/m^{m+1}\\[10pt]
(m\, |x|)^{1/m} & \text{if } 1/m^{m+1} < |x| \le 1/m\\[10pt]
1 & \text{if } |x| > 1/m,
\end{cases}
\]
set
\[
v_m(x) = \eta_m(x)\, v(x), \quad v \in E,
\]
and let $E_m = \set{\pi(v_m) : v \in E}$, where $\pi : W^{1,N}_0(\Omega) \setminus \set{0} \to \M,\, u \mapsto u/\norm{u}$ is the radial projection onto $\M$.

\begin{lemma} \label{Lemma 3.3}
There exists a constant $C = C(N,\Omega,k) > 0$ such that for all sufficiently large $m$,
\begin{enumroman}
\item $\Psi(w) \le \lambda_k(N) + \dfrac{C}{m^{N-1}} \quad \forall w \in E_m$,
\item \label{Lemma 3.3.ii} $E_m \cap \Psi_{\lambda_{k+1}(N)} = \emptyset$,
\item \label{Lemma 3.3.iii} $i(E_m) = k$.
\end{enumroman}
\end{lemma}

\begin{proof}
Let $v \in E$ and let $w = \pi(v_m)$. We have
\[
\int_\Omega |\nabla v_m|^N\, dx \le \int_{\Omega \setminus B_{1/m}(0)} |\nabla v|^N\, dx + \sum_{j=0}^N \binom{N}{j} \int_{B_{1/m}(0)} \eta_m^{N-j}\, |\nabla v|^{N-j}\, |v|^j\, |\nabla \eta_m|^j\, dx.
\]
Since $E$ is bounded in $C^1(B_{1/m}(0))$, $\nabla v$ and $v$ are bounded in $B_{1/m}(0)$. Clearly, $\eta_m \le 1$, and a direct calculation shows that
\[
\int_{B_{1/m}(0)} |\nabla \eta_m|^j\, dx \le \frac{C}{m^{N-1}}, \quad j = 0,\dots,N.
\]
Since $E_m \subset \M$, it follows that
\[
\int_\Omega |\nabla v_m|^N\, dx \le 1 + \frac{C}{m^{N-1}}.
\]
Next
\[
\int_\Omega |v_m|^N\, dx \ge \int_{\Omega \setminus B_{1/m}(0)} |v|^N\, dx = \int_\Omega |v|^N\, dx - \int_{B_{1/m}(0)} |v|^N\, dx \ge \frac{1}{\lambda_k(N)} - \frac{C}{m^N}
\]
since $E \subset \Psi^{\lambda_k(N)}$. So
\[
\Psi(w) = \frac{\dint_\Omega |\nabla v_m|^N\, dx}{\dint_\Omega |v_m|^N\, dx} \le \lambda_k(N) + \frac{C}{m^{N-1}}
\]
if $m$ is sufficiently large. Taking $m$ so large that $\lambda_k(N) + C/m^{N-1} < \lambda_{k+1}(N)$ then gives \ref{Lemma 3.3.ii}. Since $E_m \subset \M \setminus \Psi_{\lambda_{k+1}(N)}$ by \ref{Lemma 3.3.ii},
\[
i(E_m) \le i(\M \setminus \Psi_{\lambda_{k+1}(N)}) = k
\]
by the monotonicity of the index and \eqref{1.5}. On the other hand, since $E \to E_m,\, v \mapsto \pi(v_m)$ is an odd continuous map,
\[
i(E_m) \ge i(E) = k.
\]
So $i(E_m) = k$.
\end{proof}

We are now ready to prove Theorem \ref{Theorem 1.2}.

\begin{proof}[\bf Proof of Theorem \ref{Theorem 1.2}]
We have $\lambda_k(N) < \lambda < \lambda_{k+1}(N)$ for some $k \in \N$. Fix $\lambda_k(N) < \lambda' < \lambda$ and $m$ so large that the conclusions of Lemma \ref{Lemma 3.3} hold with $\lambda_k(N) + C/m^{N-1} \le \lambda'$, in particular,
\begin{equation} \label{3.8}
\Psi(w) \le \lambda' \quad \forall w \in E_m.
\end{equation}
Then take $A_0 = E_m$ and $B_0 = \Psi_{\lambda_{k+1}(N)}$, and note that $A_0$ and $B_0$ are disjoint nonempty closed symmetric subsets of $\M$ such that
\[
i(A_0) = i(\M \setminus B_0) = k
\]
by Lemma \ref{Lemma 3.3} \ref{Lemma 3.3.iii} and \eqref{1.5}. Now let $R > r > 0$ and let $A$ and $B$ be as in Theorem \ref{Theorem 1.5}.

First we show that $\inf \Phi(B) > 0$ if $r$ is sufficiently small. Since $e^t \le 1 + te^t$ for all $t > 0$,
\[
F(t) \le \frac{|t|^N}{N} + t_+^\mu\, e^{\, t_+^{N'}} \quad \forall t \in \R,
\]
where $\mu = N + N' > N$. So for $u \in \Psi_{\lambda_{k+1}(N)}$,
\begin{multline*}
\Phi(ru) \ge \int_\Omega \left[\frac{r^N}{N}\, |\nabla u|^N - \frac{\lambda r^N}{N}\, |u|^N - \lambda r^\mu\, u_+^\mu\, e^{\, r^{N'}\! u_+^{N'}}\right] dx\\[10pt]
\ge \frac{r^N}{N} \left(1 - \frac{\lambda}{\lambda_{k+1}(N)}\right) - \lambda r^\mu\, \bigg(\int_\Omega e^{\, 2r^{N'}\! u_+^{N'}} dx\bigg)^{1/2} \pnorm[2 \mu]{u_+}^\mu.
\end{multline*}
If $2\, r^{N'} \le \alpha_N$, then
\[
\int_\Omega e^{\, 2r^{N'}\! u_+^{N'}} dx \le \int_\Omega e^{\, \alpha_N\, u_+^{N'}} dx,
\]
which is bounded by \eqref{1.3}. Since $W^{1,N}_0(\Omega) \hookrightarrow L^{2 \mu}(\Omega)$ and $\lambda < \lambda_{k+1}(N)$, it follows that $\inf \Phi(B) > 0$ if $r$ is sufficiently small.

Since $e^t \ge 1 + t$ for all $t > 0$,
\begin{equation} \label{3.9}
F(t) \ge \frac{|t|^N}{N} + \frac{t_+^\mu}{\mu} \quad \forall t \in \R,
\end{equation}
so for all $w \in E_m$ and $t \ge 0$,
\begin{equation} \label{3.10}
\Phi(tw) \le \int_\Omega \left[\frac{t^N}{N}\, |\nabla w|^N - \frac{\lambda t^N}{N}\, |w|^N\right] dx = \frac{t^N}{N} \left(1 - \frac{\lambda}{\Psi(w)}\right) \le - \frac{t^N}{N} \left(\frac{\lambda}{\lambda'} - 1\right) \le 0
\end{equation}
by \eqref{3.8}.

Next we show that
\[
\sup_{w \in E_m,\; s,\, t \ge 0}\, \Phi(sw + tv_0) < \frac{\alpha_N^{N-1}}{N}
\]
for a suitably chosen $v_0 \in \M \setminus E_m$. Let
\[
v_j(x) = \frac{1}{\omega_{N-1}^{1/N}}\, \begin{cases}
(\log j)^{(N-1)/N} & \text{if } |x| \le 1/j\\[10pt]
\dfrac{\log |x|^{-1}}{(\log j)^{1/N}} & \text{if } 1/j < |x| \le 1\\[10pt]
0 & \text{if } |x| > 1.
\end{cases}
\]
Then $v_j \in W^{1,N}(\R^N)$, $\norm{v_j} = 1$, and $\pnorm[N]{v_j}^N = \O(1/\log j)$ as $j \to \infty$. We take
\[
v_0(x) = \widetilde{v}_j(x) := v_j\!\left(\frac{x}{r_m}\right)
\]
with $r_m = 1/2\, m^{m+1}$ and $j$ sufficiently large. Since $B_{r_m}(0) \subset \Omega$, $\widetilde{v}_j \in W^{1,N}_0(\Omega)$ and $\norm{\widetilde{v}_j} = 1$. For sufficiently large $j$,
\[
\Psi(\widetilde{v}_j) = \frac{1}{r_m^N\, \pnorm[N]{v_j}^N} \ge \lambda
\]
and hence $\widetilde{v}_j \notin E_m$ by \eqref{3.8}. For $w \in E_m$ and $s,\, t \ge 0$,
\[
\Phi(sw + t \widetilde{v}_j) = \Phi(sw) + \Phi(t \widetilde{v}_j)
\]
since $w = 0$ on $B_{r_m}(0)$ and $\widetilde{v}_j = 0$ on $\Omega \setminus B_{r_m}(0)$. Since $\Phi(sw) \le 0$ by \eqref{3.10}, it suffices to show that
\[
\sup_{t \ge 0}\, \Phi(t \widetilde{v}_j) < \frac{\alpha_N^{N-1}}{N}
\]
for arbitrarily large $j$. Since $\Phi(t \widetilde{v}_j) \to - \infty$ as $t \to \infty$ by \eqref{3.9}, there exists $t_j \ge 0$ such that
\begin{equation} \label{3.11}
\Phi(t_j \widetilde{v}_j) = \frac{t_j^N}{N} - \lambda \int_{B_{r_m}(0)} F(t_j \widetilde{v}_j)\, dx = \sup_{t \ge 0}\, \Phi(t \widetilde{v}_j)
\end{equation}
and
\begin{equation} \label{3.12}
\Phi'(t_j \widetilde{v}_j)\, \widetilde{v}_j = t_j^{N-1} \left(1 - \lambda \int_{B_{r_m}(0)} \widetilde{v}_j^N e^{\, t_j^{N'} \widetilde{v}_j^{N'}} dx\right) = 0.
\end{equation}
Suppose $\Phi(t_j \widetilde{v}_j) \ge \alpha_N^{N-1}/N$ for all sufficiently large $j$. Since $F(t) \ge 0$ for all $t \in \R$, then \eqref{3.11} gives $t_j^{N'} \ge \alpha_N$, and then \eqref{3.12} gives
\begin{multline*}
\frac{1}{\lambda} = \int_{B_{r_m}(0)} \widetilde{v}_j^N e^{\, t_j^{N'} \widetilde{v}_j^{N'}} dx \ge \int_{B_{r_m}(0)} \widetilde{v}_j^N e^{\, \alpha_N\, \widetilde{v}_j^{N'}} dx\\[10pt]
= r_m^N \int_{B_1(0)} v_j^N e^{\, \alpha_N\, v_j^{N'}} dx \ge r_m^N \int_{B_{1/j}(0)} v_j^N e^{\, \alpha_N\, v_j^{N'}} dx = \frac{r_m^N}{N}\, (\log j)^{N-1},
\end{multline*}
which is impossible for large $j$.

Now we show that $\Phi \le 0$ on $A$ if $R$ is sufficiently large. In view of \eqref{3.10}, it only remains to show that $\Phi(Ru) \le 0$ for $u = \pi((1 - t)\, w + tv_0),\, w \in E_m,\, 0 \le t \le 1$. Since
\[
\norm{(1 - t)\, w + tv_0} \le (1 - t) \norm{w} + t \norm{v_0} = 1
\]
and $w$ and $v_0$ are supported on disjoint sets, we have
\begin{equation} \label{3.13}
\pnorm[N]{u}^N = \frac{\pnorm[N]{(1 - t)\, w + tv_0}^N}{\norm{(1 - t)\, w + tv_0}^N} \ge (1 - t)^N \pnorm[N]{w}^N + t^N \pnorm[N]{v_0}^N \ge \frac{(1 - t)^N}{\Psi(w)} \ge \frac{(1 - t)^N}{\lambda'}
\end{equation}
by \eqref{3.8}, and
\begin{equation} \label{3.14}
\pnorm[\mu]{u_+}^\mu = \frac{\pnorm[\mu]{[(1 - t)\, w + tv_0]_+}^\mu}{\norm{(1 - t)\, w + tv_0}^\mu} \ge (1 - t)^\mu \pnorm[\mu]{w_+}^\mu + t^\mu \pnorm[\mu]{v_0}^\mu \ge t^\mu \pnorm[\mu]{v_0}^\mu.
\end{equation}
By \eqref{3.9}, \eqref{3.13}, and \eqref{3.14},
\[
\Phi(Ru) \le \frac{R^N}{N} \norm{u}^N - \frac{\lambda R^N}{N} \pnorm[N]{u}^N - \frac{R^\mu}{\mu} \pnorm[\mu]{u_+}^\mu \le - \frac{1}{N} \left[\frac{\lambda}{\lambda'}\, (1 - t)^N - 1\right] R^N - \frac{1}{\mu} \pnorm[\mu]{v_0}^\mu\, t^\mu R^\mu.
\]
The last expression is clearly nonpositive if $t \le 1 - (\lambda'/\lambda)^{1/N} =: t_0$. For $t > t_0$, it is nonpositive if $R$ is sufficiently large.

The inequalities \eqref{1.6} now imply that $0 < c < \dfrac{\alpha_N^{N-1}}{N}$. If $\Phi$ has no \C{c} sequences, then $\Phi$ satisfies the \C{c} condition trivially and hence $c$ is a critical value of $\Phi$ by Theorem \ref{Theorem 1.5}. If $\Phi$ has a \C{c} sequence, then a subsequence converges weakly to a nontrivial critical point of $\Phi$ by Lemma \ref{Lemma 3.2}.
\end{proof}

\textbf{Acknowledgement:} This work was completed while the first- and second-named authors were visiting the Academy of Mathematics and Systems Science at the Chinese Academy of Sciences, and they are grateful for the kind hospitality of the host institution.

\def\cdprime{$''$}

\end{document}